\documentclass[12pt,reqno]{amsart}

\newcommand\version{March 5, 2014}

\usepackage{amsmath,amsfonts,amsthm,amssymb,amsxtra}

\newtheorem{theorem}{Theorem}[section]
\newtheorem{proposition}[theorem]{Proposition}

\newtheorem{corollary}[theorem]{Corollary}

\theoremstyle{definition}

\theoremstyle{remark}

\numberwithin{equation}{section}

\renewcommand{\epsilon}{\varepsilon}

\renewcommand{\phi}{\varphi}
\newcommand{\R}{\mathbb{R}}

\begin{document}

\title[A smooth, complex generalization of the Hobby-Rice
theorem
\version]{A smooth, complex generalization of the Hobby-Rice
theorem}

\author{Oleg Lazarev}
\address{Oleg Lazarev, Department of Mathematics,
Princeton University, Washington Road, Princeton, NJ, 08542, USA.}
\address{Present address: 
Department of Mathematics,
Stanford University, Building 380, Stanford, CA 94305, USA}
\email{olazarev@stanford.edu}

\author{Elliott H. Lieb}
\address{Elliott H. Lieb, Departments of Mathematics and Physics,
Princeton University,  Washington Road, Princeton, NJ
08544, USA} \email{lieb@princeton.edu} 
\date{March 27, 2014}

\thanks{\copyright\, 2012 by the authors. This paper may be
reproduced, in its entirety, for non-commercial purposes.\\
Work partially supported by NSF grant PHY--0965859
(E.H.L.)}

\thanks{Published in Indiana Univ. Math. Jour. {\bf 62}, 1133-1141 (2013).
DOI: 10.1512/iumj.2013.62.5062.}

\begin{abstract}

The Hobby-Rice Theorem states that, given $n$ functions $f_j$ on
$\R^N$, there exists
a multiplier $h$ such that the integrals of $f_jh$ are all simultaneously
zero. This multiplier takes values~$\pm1$ and is discontinuous. We show how
to find a multiplier $h=e^{ig}$ that is infinitely differentiable, takes values on the
unit circle,
and is such that the integrals of $f_jh$ are all zero.
We also show the existence of $n$ infinitely differentiable,
real functions $g_j$ such that the 
$n$ functions $f_j e^{ig_j}$  are pairwise orgthogonal.

\end{abstract}

\maketitle

\section{Introduction}

The purpose of this paper is to generalize the Hobby-Rice Theorem \cite{HR}
(see also the later, shorter proof by Pinkus \cite{P}), which states that,
given $n$ integrable real-valued functions $f_1,\dots,f_n$ on $[0,1]$, it is
possible to find $0=\alpha_0<\alpha_1<\dots<\alpha_r<\alpha_{r+1}=1$ (with
$r\le n$) such that
\begin{equation}\label{eqn: 01}
\sum_{m=1}^{r+1}(-1)^m  \int_{\alpha_{m-1}}^{\alpha_m}f_j(x)dx=0
\end{equation}
for \emph{all} $1\le j\le n$.

That is, the Hobby-Rice Theorem says that there exists
$g:[0,1]\mapsto\{0,\pi\}$ such that
\begin{equation}\label{eqn: 02}
\int_0^1f_j(x)e^{ig(x)}dx=0
\end{equation}
for $1\le j\le n$. This $g$ is discontinuous, however. The purpose of this
paper is to generalize the Hobby-Rice Theorem by letting $g\in
C_c^{\infty}((0,1);\mathbb{R})$. In particular, this implies that if
$f_1,\dots,f_n\in H^1([0,1])$, then $f_j e^{ig}\in H^1([0,1])$ as well.
\smallskip

If the functions
$f_j$ are complex-valued, we can consider them to be $2n$ real-valued
functions; thus, \eqref{eqn: 01}, \eqref{eqn: 02} continue to hold except
that $r\le2n$. The words `complex generalization' in our title do \emph{not}
refer to this trivial generalization, but rather, to the fact that the
multiplier is a smooth function on the unit circle in the complex plane.

There seems to be no way to adapt the proof of the Hobby-Rice Theorem (which
involves a fixed-point argument) to find a smooth $g$. Instead, we start
with the discontinuous $g$ given by the Hobby-Rice Theorem, and modify $g$
to get a smooth function. Our proof is more complicated than desired, and we
offer it as a challenge for simplification. In particular, we have no
estimate on the derivatives of $g$, and therefore, we cannot estimate the
$H^1$ norm of $f_j e^{ig}$ in terms of the $H^1$ norm of $f_j$. 
Vermont  Rutherfoord has extended our results in \cite{R}.

Our main theorem is the following.

 \begin{theorem}\label{thm: 1}
Let $f_1,\dots,f_n\in L^1([0,1])$ be real-valued. Then, there exists $g\in
C_c^{\infty}(0,1)$, real-valued, such
that
\begin{equation}\label{eqn: 0}
\int_0^1f_j(x)e^{ig(x)}dx=0
\end{equation}
for $1\le j\le n$.
\end{theorem}
As noted before, the validity of Theorem \ref{thm: 1} for real-valued
functions implies its validity for complex-valued functions. The following
is also an important corollary of Theorem \ref{thm: 1}.

\begin{corollary}\label{cor: 1}
If the functions in Theorem \ref{thm: 1} are taken to be in
$L^1(\mathbb{R}^N)$ instead of $L^1([0,1])$, then there exists a real-valued
$g\in C^{\infty}_c(\mathbb{R}^1)$
so that
\begin{equation}\label{eqn: 00}
\int_{\mathbb{R}^N}f_j(x)e^{ig(x_1)}dx=0
\end{equation}
for $1\le j\le n$.
\end{corollary}

 \begin{proof}[Proof of the corollary]
First, to reduce the $L^1(\mathbb{R}^1)$ version to the $L^1([0,1])$
version, we map $x$ to $(e^{-x}+1)^{-1}$ so that $(-\infty,\infty)$ maps to
$[0,1]$, and we absorb the Jacobian into the functions on $[0,1]$. That is,
we have
\begin{equation}
\int_{\mathbb{R}^1}f_j(x)\exp(ig((e^{-x}+1)^{-1}))dx=\int_0^1\frac{
f_j(-\ln(x^{-1}-1))}{x(1-x)}e^{ig(x)}dx.
\end{equation}
We then apply Theorem \ref{thm: 1} to $f_j(-\ln(x^{-1}-1))/(x(1-x))$ to find
$g(x)$. The solution to the problem on $\mathbb{R}^1$ will then be
$g((e^{-x}+1)^{-1})$.

To reduce the $L^1(\mathbb{R}^N)$ version to the $L^1(\mathbb{R}^1)$
version, we let
\begin{equation}
F_j(x)=\int_{\mathbb{R}^{N-1}}f_j(x,x_2,\dots,x_N)dx_2\cdots dx_N,
\end{equation}
and then apply the $L^1(\mathbb{R}^1)$ version of Theorem \ref{thm: 1} to
$F_j(x)$, obtaining $g(x)$.
$g(x_1,\dots,x_N)=g(x_1)$. This function does not have compact support in
$\mathbb{R}^N$, but is in $C^{\infty}(\mathbb{R}^N)$ and is bounded.
\end{proof}

Theorem \ref{thm: 1} and Corollary \ref{cor: 1} can be used as follows to
smoothly orthogonalize any set of any $n$ functions without changing their
moduli.
\begin{corollary}\label{cor: 3}
Let $f_1, \ldots, f_n\in L^2(\mathbb{R}^N)$. Then, there exist $n$ real
functions $g_1, \ldots, g_n \in {C}^\infty_c (\mathbb{R}^1)$ such that 
the $n$ functions $\phi_j(x) := f_j(x) \exp\{i g_j(x_1)\}$ are pairwise
orthogonal.
\end{corollary}

\begin{proof}[Proof of Corollary]
Use Theorem \ref{thm: 1} as follows.

Choose $g_n(x) =0$. Choose $g_{n-1} $ so that  $\int f_{n-1}^* \phi_n
e^{-ig_{n-1}} =0$. Next, choose $g_{n-2}$ so that the two integrals $\int
f_{n-2}^* \phi_j e^{-ig_{n-2}} =0 $ for $j= n$ and $j= n-1$. Similarly,
determine $g_{n-3}$ so that three integrals vanish, and so on, finishing
with  $\int  f_1^* \phi_j e^{-ig_{1}}  =0$ for $j=2,\dots n$.
\end{proof}

 \subsection{Acknowledgements.}
The authors are grateful to Rupert Frank for helpful comments and to Robert
Schrader for motivating our work and for his constant encouragement. The
motivation for our work is to prove a necessary ingredient in a proof of a
theorem in the `density functional theory' of quantum mechanics \cite{LS}.
Specifically, we want to know that, given a one-particle density and a
one-particle current density, we can find $n$ orthonormal functions in
$H^1(\mathbb{R}^3)$ such that the sums of their individual densities and
currents equal the given values.


\section{Proof of Theorem \ref{thm: 1}}

To prove Theorem \ref{thm: 1}, we start with the following special case of
that theorem, which will be used in an inductive proof.

\begin{theorem}\label{thm: 2}
Suppose the conditions of Theorem \ref{thm: 1} hold, but with the additional
assumption that there exists $p\in(0,1)$ such that the functions
$f_1,\dots,f_n$ are linearly independent on $[0,p]$ and on $[p,1]$. Then,
the conclusion of Theorem \ref{thm: 1} is true.
\end{theorem}
We begin by proving Theorem \ref{thm: 2}.

\begin{proof}
By the Hobby-Rice theorem for $f_1,\dots,f_n$ restricted to $[0,p]$, there
exist $\{\alpha_m\}_{m=1}^r$, $r\le n$,
$0=\alpha_0<\alpha_1<\dots<\alpha_r<\alpha_{r+1}=p$ such that, for $1\le
j\le n$,
\begin{equation}
\sum_{m=1}^{r+1}(-1)^m\int_{\alpha_{m-1}}^{\alpha_m}f_j(x)dx=0.
\end{equation}
Similarly, by the Hobby-Rice theorem for $f_1,\dots,f_n$ restricted to
$[p,1]$, there exist $\{\beta_m\}_{m=1}^s,\ s\le n,\
p=\beta_0<\beta_1<\dots<\beta_s<\beta_{s+1}=1$ such that, for $1\le j\le n$,
\begin{equation}
\sum_{m=1}^{s+1}(-1)^m\int_{\beta_{m-1}}^{\beta_m}f_j(x)dx=0.
\end{equation}
Therefore, we have
\begin{equation}\label{eqn: 1}
0=\bigg(\sum_{m=1}^{r+1}
(-1)^m\int_{\alpha_{m-1}}^{\alpha_m}f_j(x)\\dx\bigg)+i\bigg(\sum_{m=1}^{s+1}
(-1)^m\int_{\beta_{m-1}}^{\beta_m}f_j(x)dx\bigg)
\end{equation}
for $1\le j\le n$. Writing the $\pm1,\pm i$ in complex polar coordinates, we
have
\begin{equation}\label{eqn: 22}
\int_0^1f_j(x)e^{ig_0(x)}dx=0,
\end{equation}
where
\begin{equation}
g_0(x)=\begin{cases}
0&:x\in[\alpha_m,\alpha_{m+1}),\ m\mbox{ odd},\\
\pi&:x\in[\alpha_m,\alpha_{m+1}),\ m\mbox{ even},\\
\pi/2&:x\in[\beta_m,\beta_{m+1}),\ m\mbox{ odd},\\
3\pi/2&:x\in[\beta_m,\beta_{m+1}),\ m\mbox{ even}.\\
\end{cases}
\end{equation}
Note that $e^{ig_0}$ is real to the left of $p$ and imaginary to the right
of $p$.

Even though the integrals in \eqref{eqn: 22} are zero, the problem is that
$g_0$ is not smooth (or even continuous) at the points $\alpha_m$ or
$\beta_m$.
Therefore, for $\epsilon>0$, we now consider an approximation
$g_{\epsilon}\in C_c^{\infty}((0,1); \mathbb{R})$ to $g_0$ such that
$g_{\epsilon}$ differs from $g_0$ only in $\epsilon$-intervals around the
points $\alpha_m$ and $\beta_m$,
and moreover, such that
$0\le g_{\epsilon}(x)<2\pi$. Since the functions $f_j$
are integrable, for $1\le j\le n$
\begin{equation}\label{eqn: 2}
\bigg|\int_0^1f_j(x) e^{i
g_{\epsilon}(x)}dx\bigg|=\bigg|\int_0^1f_j(x)(e^{ig_{\epsilon}(x)}-e^{
ig_0(x)})dx\bigg|\to0
\quad\mbox{ as }\epsilon\to0.
\end{equation}

We would like to replace $g_0$ by $g_{\epsilon}$, which will be smooth, but
then the integrals in \eqref{eqn: 22} will no longer be equal to zero.
Therefore, we shall further perturb $g_{\epsilon}$ by smooth functions so as
to make the integrals in \eqref{eqn: 22} zero again.

To do so, consider functions $h_1$, $\dots$, $h_n\in
C_c^{\infty}((0,p);\mathbb{R})$ and $h_{n+1}$, $\dots$, $h_{2n}\in
C_c^{\infty}((p,1);\mathbb{R})$. For such functions $h$,
define $F_{\epsilon}^j:\mathbb{R}^{2n}\mapsto\mathbb{C}$ by
\begin{equation}
F_{\epsilon}^j(u)=\int_0^1f_j(x)\exp\Big(ig_{\epsilon}(x)+i\sum_{m=1}^{2n}
u_mh_m(x)\bigg)dx
\end{equation}
where $u=(u_1,\dots,u_{2n})\in\mathbb{R}^{2n}$. Also, define
$Q_{\epsilon}:\mathbb{R}^{2n}\mapsto\mathbb{R}^{2n}$ by
\begin{equation}
Q_{\epsilon}(u)=(\Im(F_{\epsilon}^1(u)),\dots,\Im(F_{\epsilon}^n(u)),\Re(F_{
\epsilon}^1(u)),\dots,\Re(F_{\epsilon}^n(u))).
\end{equation}
Note that $Q_0(0)=0$, because of \eqref{eqn: 22}. If $\|\cdot\|$ is the
Euclidean norm on $\mathbb{R}^{2n}$, then, by \eqref{eqn: 2},
$\|Q_{\epsilon}(0)\|\to0$ as $\epsilon\to0$.

To show that we can find $u$ to compensate the $g_0$ to $g_{\epsilon}$
perturbation, we start by showing that $u\mapsto Q_0(u)$ is invertible in a
neighborhood of $u=0$ for \emph{some} $h_1,\dots,h_{2n}$. To do this, we use
the inverse function theorem, which requires showing that the Jacobian
matrix $DQ_0(0)$ has a non-zero determinant. Note that the partial
derivatives are
\begin{equation}\label{eqn3}
\frac{\partial F_0^j}{\partial
u_k}=i\int_0^1f_j(x)h_k(x)\exp\Big(ig_0(x)+i\sum_{m=1}^{2n}u_mh_m(x)\Big)dx
,
\end{equation}
which are continuous in $u$ on $\mathbb{R}^{2n}$. In particular, at $u=0$,
the partial derivatives are
\begin{equation}
\frac{\partial F_0^j}{\partial
u_k}\bigg|_{u=0}=i\int_0^1f_j(x)h_k(x)e^{ig_0(x)}dx.
\end{equation}

Also, note that we constructed $g_0$ so that $ie^{g_0}$ is imaginary in
$[0,p]$, the support of $h_1,\dots,h_n$, and, similarly, so that $ie^{g_0}$
is real in $[p,1]$, the support of $h_{n+1},\dots,h_{2n}$. Therefore,
$(\partial F_0^j/\partial u_k)|_{u=0}$ is imaginary for $1\le k\le n$ and
real for $n+1\le k\le2n$. As a result, the Jacobian matrix is
block-diagonal:
\[
DQ_0(0)=\begin{pmatrix}A&0\\0&B\\\end{pmatrix},
\]
where $A,B$ are $n\times n$ matrices with entries
\begin{align}
A_{jk}&=\Im\left(\frac{\partial F_0^j}{\partial
u_k}\right)\bigg|_{u=0}=\int_0^pf_j(x)h_k(x)e^{ig_0(x)}dx\\
B_{jk}&=\Re\left(\frac{\partial F_0^j}{\partial
u_k}\right)\bigg|_{u=0}=i\int_p^1f_j(x)h_{k+n}(x)e^{ig_0(x)}dx.
\end{align}

By Proposition \ref{prop: 1}, which we will prove later, there do exist
real-valued functions
$h_1$, $\dots$, $h_n\in C_c^{\infty}((0,p);\mathbb{R})$ and $h_{n+1}$,
$\dots$, $h_{2n}\in C_c^{\infty}((p,1);\mathbb{R})$ such that
$\det{DQ_0(0)}\ne0$. Proposition \ref{prop: 1} is where we use the condition
of Theorem \ref{thm: 2} that the functions $f$
are linearly independent in $[0,p]$ and in $[p,1]$. Also, since the
determinant of a matrix is a continuous function of the matrix entries, we
may assume that there exists $\epsilon_0>0$ such that the functions $h$
vanish in $\epsilon_0$-intervals of the points $\alpha_m$ and $\beta_m$,
and that $\det{DQ_0(0)}$ is still not equal to zero.
While this requirement may seem like a large perturbation of the functions
$h$,
it is a small perturbation of integrals defining $A_{jk}$ and $B_{jk}$.

The crucial point about our construction is that for $\epsilon<\epsilon_0$,
the supports of $g_{\epsilon}-g_0$ and the functions $h$
are now disjoint! By using this fact, we will show that $Q_{\epsilon}$ can
be written as
\begin{equation}\label{eqn: 4}
Q_{\epsilon}(u)=Q_0(u)+C(\epsilon),\quad\mbox{for }\epsilon<\epsilon_0,
\end{equation}
where $C(\epsilon)$ is independent of $u$ and $\|C(\epsilon)\|\to0$ as
$\epsilon\to0$. Thus, we have decoupled the smoothing of the jump in $g_0$
(which is the $C(\epsilon)$ term) from the perturbation of constant parts of
$g_0$ (which is the $Q_0(u)$ term). It is then easy to use the latter to
compensate for the former.

To prove \eqref{eqn: 4}, we rewrite $F_j^{\epsilon}$ as
\begin{equation}
F_j^{\epsilon}(u)=F_j^0(u)+(F_j^{\epsilon}(u)-F_j^0(u)).
\end{equation}
Note that $F_j^{\epsilon}(u)-F_j^0(u)$ does not depend on $u$ when
$\epsilon<\epsilon_0$. This is because
\begin{equation}
F_j^{\epsilon}(u)-F_j^0(u)=\int_0^1f_j(x)\exp\Big(i\sum_{m=1}^{2n}
u_mh_m(x)\Big)(e^{ig_{\epsilon}(x)}-e^{ig_0(x)})dx,
\end{equation}
and $e^{i g_{\epsilon}(x)} - e^{i g_0(x)}$ is non-zero only on
$\epsilon$-intervals centered at the points $\alpha_m$ and $\beta_m$. On the
other hand, the functions $h$ vanish on $\epsilon_0$-intervals centered at
the points $\alpha_m$ and $\beta_m$.
Therefore, if $C(\epsilon)\in\mathbb{R}^{2n}$ denotes the real and imaginary
parts of $F_j^{\epsilon}(u)-F_j^0(u)$, we have \eqref{eqn: 4}. Also, note
that we have $\|C(\epsilon)\|\to0$ as $\epsilon\to0$ because of \eqref{eqn:
2}.

We next explain how to use the inverse function theorem to carry out the
compensation. Since $Q_0(0)=0$ and $\det DQ_0(0)\ne0$ and the partial
derivatives of $Q_0$ are continuous, the theorem guarantees the existence of
a ball $B_R(0)\subset\mathbb{R}^{2n}$ of radius $R>0$ around $0$ such that,
when $y\in B_R(0)$, there exists $u\in\mathbb{R}^{2n}$ such that $Q_0(u)=y$.

As stated in \eqref{eqn: 4}, for $\epsilon<\epsilon_0$, we see that
$Q_{\epsilon}$ and $Q_0$ differ by a constant depending only on $\epsilon$.
Thus, for every $y\in B_R(Q_{\epsilon}(0))$, there exists
$u\in\mathbb{R}^{2n}$ such that $Q_{\epsilon}(u)=y$. Now take
$0<\epsilon<\epsilon_0$ small enough so that $\|C(\epsilon)\|<R$. Then
$Q_{\epsilon}(0)=C(\epsilon)$ and $\|Q_{\epsilon}(0)-0\|=\|C(\epsilon)\|<R$.
Therefore, $0\in B_R(Q_{\epsilon}(0)),$ and so there exists $u\in
\mathbb{R}^{2n}$ such that $Q_{\epsilon}(u)=0$. For such $u$, let
$g(x)=g_{\epsilon}(x)+\sum_{m=1}^{2n}u_mh_m(x)$. Then,
\begin{equation}
\int_0^1f_j(x)e^{ig(x)}dx=0
\end{equation}
for $1\le j\le n$, as desired. Note that $g\in
C_c^{\infty}((0,1);\mathbb{R})$ since $g_{\epsilon}\in
C_c^{\infty}((0,1);\mathbb{R})$ for $\epsilon>0$, and we chose $h_k\in
C_c^{\infty}((0,p);\mathbb{R})$ and $h_{n+k}\in
C_c^{\infty}((p,1);\mathbb{R})$.
\end{proof}
To complete the proof of Theorem \ref{thm: 2}, it is necessary to prove
Proposition \ref{prop: 1}, which was used in the proof above.

\begin{proposition}\label{prop: 1}
If the functions satisfy the conditions of Theorem \ref{thm: 2}, there exist
real-valued $h_1,\dots,h_n\in C_c^{\infty}((0,p);\mathbb{R})$,
$h_{n+1},\dots,h_{2n}\in C_c^{\infty}((p,1);\mathbb{R})$ such that
$\det{DQ_0(0)}\ne0$.
\end{proposition}

\begin{proof}
Our proof is by contradiction. Suppose that $\det{DQ_0(0)}=0$ for every
$h_1,\dots,h_n\in C_c^{\infty}((0,p);\mathbb{R})$, $h_{n+1},\dots,h_{2n}\in
C_c^{\infty}((p,1);\mathbb{R})$. Then, either $\det{A}=0$ for every $h_1$,
$\dots$, $h_n$ or $\det{B}=0$ for every $h_{n+1},\dots,h_{2n}$; otherwise,
if $\det{A}\ne0$ for some $h_1,\dots,h_n$ and $\det{B}\ne0$ for some
$h_{n+1},\dots,h_{2n}$, then $\det{DQ_0(0)}=\det{A}\det{B}\ne0$ for such
$h_1,\dots,h_{2n}$, which is a contradiction. Having assumed linear
independence of the functions $f$
both to the left of $p$ and to the right of $p$, let us assume, without loss
of generality, that $\det{A}=0$ for every $h_1,\dots,h_n$. We will show that
there exist $a_1,\dots,a_n\in\mathbb{R}$ (not all $0$) such that
$\sum_{j=1}^na_jf_j(x)=0$ almost everywhere on $[0,p]$.

Suppose we let $h_1$ vary and fix the rest of the functions $h$.
Then, for all $h_1$, we have
\begin{equation}\label{eqn4}
\sum_{j=1}^nC_{j,1}\int_0^pf_j(x)h_1(x)e^{ig_0(x)} dx=\det{A}=0,
\end{equation}
where $C_{j,1}$ is the $(j,1)$-cofactor of $A$. We can assume that at least
one of the cofactors is non-zero, since otherwise we can look at a smaller
matrix. Note that these cofactors depend only on $h_2,\dots,h_n$, and not on
$h_1$. Since \eqref{eqn4} holds for all $h_1\in
C_c^{\infty}((0,p);\mathbb{R})$,
\begin{equation}\label{eqn5}
\sum_{j=1}^nC_{j,1}f_j(x)=0
\end{equation}
almost everywhere on $[0,p]$, which contradicts independence.
\end{proof}
We now use Theorem \ref{thm: 2} to complete the proof of Theorem \ref{thm:
1}. First, we use the following proposition to split the theorem into two
cases.

\begin{proposition}\label{prop: p2}
Given $n$ Lebesgue-measurable functions on $[0,1]$, exactly one of the two
following cases holds:
\begin{enumerate}
\item There exists $p\in(0,1)$ such that $f_1,\dots,f_n$ are linearly
independent
almost everywhere on $[0,p]$ and on $[p,1]$.
\item There exists $q\in[0,1]$ such that $f_1,\dots,f_n$ are linearly
dependent almost everywhere on $[0,q]$ and on $[q,1]$.
\end{enumerate}
\end{proposition}

\begin{proof}
To see that one of these cases holds, consider
\begin{align*}
L&\!=\sup\Big\{x_0\in[0,1]\!:\!\!\sum_{j=1}^na_jf_j(x)=0\mbox{ a.e. on
$[0,x_0]$ for some }(a_1,\dots,a_n)\ne0\Big\}\\
R&\!=\inf\Big\{x_0\in[0,1]\!:\!\!\sum_{j=1}^nb_jf_j(x)=0\mbox{ a.e. on
$[x_0,1]$ for some }(b_1,\dots,b_n)\ne0\Big\}.
\end{align*}
Note that the extremizers $L$
and $R$ are actually reached; in other words,
there exists $(a_1,\dots,a_n)$ such that $\sum_{j=1}^na_jf_j(x)=0$ on
$[0,L]$, and similarly for $[R,1]$. This is because, for every $x_0<L$, the
set of relations $(a_1,\dots,a_n)$ such that $\sum_{j=1}^na_jf_j(x)=0$ on
$[0,x_0]$ is a vector space whose dimension is positive and finite. As $x$
increases, this dimension can only decrease, since if
$\sum_{j=1}^na_jf_j(x)=0$ on $[0,x_0]$, then this is also true on $[0,y]$
for any $y<x_0$. Now consider the relations that form a basis for this
vector space. Since the dimension of the vector space only decreases, one of
these relations must hold until $L$, which means that a relationship holds
almost everywhere on $[0,L]$ as desired.

To prove the claim that either Case 1 or Case 2 must hold, note that if
$L<R$, then Case 1 holds for any $p$ satisfying $L<p<R$; and if $L\ge R$,
then Case 2 holds for any $q$ satisfying $R\le q\le L$.
\end{proof}
We prove Theorem \ref{thm: 1} by induction on $n$, the number of functions.
If $n=1$, then the linear independence condition of Theorem \ref{thm: 2}
holds for some $p$ because $f_1$ is not identically zero.

Suppose that Theorem \ref{thm: 1} holds for $n-1$ functions on $[0,1]$.
Consider $n$ functions $f_1,\dots,f_n$. By Proposition \ref{prop: p2}, this
set of functions belongs either to Case 1 or Case 2. Note that Case 1 is
exactly appropriate for Theorem \ref{thm: 2}, and so Theorem \ref{thm: 1}
holds in this case.

Now suppose that the $n$ functions belong to Case 2. Therefore, there are no
more than $n-1$ linearly independent functions on both sides of $q$ (even if
$q=0$ or $1$). By the induction hypothesis, we can apply Theorem \ref{thm:
1} to $[0,q]$ and to $[q,1]$ (via change-of-variables). Then, there exist
$g_1\in C_c^{\infty}((0,q);\mathbb{R})$ and $g_2\in
C_c^{\infty}((q,1);\mathbb{R})$ such that
\begin{equation}  
\int_0^qf_j(x)e^{ig_1(x)}dx=0 \quad \mbox{for }1\le j \le n
\end{equation}
and
\begin{equation}
\int_q^1f_j(x)e^{ig_2(x)}dx =0 \quad \mbox{for }1\le j\le n.
\end{equation}
Let $g:[0,1]\mapsto[0,1]$ be such that
$g\negthickspace\upharpoonright_{[0,q]}=g_1$ and
$g\negthickspace\upharpoonright_{[q,1]}=g_2$. Then, we have
\begin{equation}
\int_0^1f_j(x)e^{ig(x)}dx=0 \quad \mbox{for }1\le j\le n
\end{equation}  
as desired. Note that $g\in C_c^{\infty}((0,1);\mathbb{R})$ because
$g_1(q)=g_2(q)=0$
and all derivatives of $g_1,g_2$ are zero at $q$ (by the compact support
condition for $g_1$ and $g_2$). Therefore, Theorem \ref{thm: 1} holds by
induction.\hfill \qedsymbol

\bibliographystyle{amsalpha}

\end{document}